\documentclass{amsart}

\usepackage{amsmath,amsfonts,amssymb,amsthm,enumerate,tikz-cd}

\def\Z{\mathbb Z}
\def\Q{\mathbb Q}

\newcommand{\dR}{\mathrm{dR}}

\newcommand{\colim}{\mathrm{colim}}

\newcommand{\rk}{\mathrm{rk}}
\newcommand{\Lie}{\mathrm{Lie}}

\newcommand{\ord}{\mathrm{ord}}

\newcommand{\val}{\mathrm{val}}

\theoremstyle{plain}
\newtheorem{theorem}{Theorem}

\newtheorem{lemma}{Lemma}
\newtheorem{proposition}{Proposition}
\newtheorem{corollary}{Corollary}

\theoremstyle{definition}

\newtheorem{definition}{Definition}

\newcommand{\Gal}{\mathrm{Gal}}
\newcommand{\nr}{\mathrm{nr}}
\newcommand{\F}{\mathbb{F}}
\newcommand{\Mat}{\mathrm{Mat}}

\newcommand{\red}{\mathrm{red}}

\newcommand{\tor}{\mathrm{tor}}
\newcommand{\divv}{\mathrm{div}}
\newcommand{\nab}{\nabla_{0}^{1}}
\usepackage[all]{xy}

\begin{document}
\title{A Buium--Coleman Bound for the Mordell--Lang conjecture}
\author[N. Dogra]{Netan Dogra}
\author[S. Pandit]{Sudip Pandit}
\maketitle

\begin{abstract}
For $X$ a hyperbolic curve of genus $g$ with good reduction at $p\geq 2g$, we give an explicit bound on the Mordell--Lang locus $X(\mathbb{C})\cap \Gamma $, when $\Gamma \subset J(\mathbb{C})$ is the divisible hull of a subgroup of $J(\Q _p ^{\mathrm{nr}})$ of rank less than $g$. Without any assumptions on the rank (but with all the other assumptions) we show that $X(\mathbb{C})\cap \Gamma $ is unramified at $p$, and bound the size of its image in $X(\overline{\F }_p )$. As a corollary, we obtain a new proof that Mordell implies Mordell--Lang for curves.
\end{abstract}
\section{Introduction}
Let $p$ be a prime number. Let $R$ be the ring of integers of the maximal unramified extension $\Q _p ^{\mathrm{nr}}$ of $\Q _p $, $X/R$ a smooth proper geometrically irreducible curve of genus $g>1$, and $\iota :X\to J$ the Abel--Jacobi morphism defined by sending a point $b\in X(\Q _p ^{\nr})$ to $0\in J(\Q _p ^{\nr })$. Given a subset $S$ of an abelian group $G$, let $S_{\divv }$ denote the divisible hull of $S$, i.e.
\[
S_{\divv}:=\{g\in G, \exists n\in \Z _{>0},n\cdot g\in S \}.
\] 
Let $\Gamma <J(\overline{\Q }_p )$ be a subgroup such that $\dim _{\Q }\Gamma \otimes _{\mathbb{Z}}\Q =r$. We say that $\Gamma \otimes _{\Z }\Q $ has a basis defined over $L$ if there is a subgroup $H\subset \Gamma \cap J(L)$ such that the map $H\otimes _{\Z }\Q \to \Gamma \otimes _{\Z }\Q $ is surjective (or equivalently $\Gamma $ is contained in the divisible hull of $J(L)$). In this paper we prove the following theorem.
\begin{theorem}\label{thm:main1}
\begin{enumerate}
\item Let $\Gamma <J(\overline{\Q }_p )$ be a subgroup such that $\dim _{\Q }\Gamma \otimes _{\mathbb{Z}}\Q =r$. If $\Gamma \otimes _{\mathbb{Z}}\Q $ has a basis defined over $\Q _p ^{\mathrm{nr}}$, then
\[
\# \red (X(\Q _p ^{\mathrm{nr}})\cap \Gamma )\leq p^{3g+r}3^g [p(2g-2)+6g]g! .
\]
where 
\[
\red :X(\overline{\Q }_p )\to X(\overline{\F }_p )
\]
denotes the reduction map.
\item If $p\geq 2g$, then
\[
X(\overline{\Q }_p )\cap J(\Q _p ^{\mathrm{nr}})_{\divv} \subset X(\Q _p ^{\mathrm{nr}}).
\]
\item If $p\geq 2g$, $\Gamma $ is as in part (1), and $r<g$, then
\[
\# X(\overline{\Q }_p )\cap \Gamma \leq p^{3g+r}3^g [p(2g-2)+6g]g! +2r.
\]
\end{enumerate}
\end{theorem}

As an application, we deduce the following.
\begin{corollary}\label{cor:1}
Let $K$ be a number field and $X/K$ a smooth projective geometrically irreducible curve of genus $g>1$ with a $K$-rational point $b$. Let $\Gamma <J(\overline{K})$ be a subgroup such that $\Gamma \otimes _{\Z }\Q $ is finite dimensional and let $\iota :X\to J$ denote the Abel--Jacobi morphism sending $b$ to $0$. Then there is a finite extension $L/K$ such that
\[
\iota (X(\overline{K}))\cap \Gamma \subset \iota (X(L)).
\]
\end{corollary}
Of course this corollary is a special case of Faltings' proof of the Mordell--Lang conjecture \cite{faltings:general}. 
Corollary \ref{cor:1} is also proved by Raynaud in the same paper in which he proves the Manin--Mumford conjecture \cite{Ray1}. Our proof has some common features with Raynaud's proof. He also proves finiteness of the image of $X(\Q _p ^{\mathrm{nr}})\cap \Gamma $ in $X(\overline{\F }_p )$, and that all but finitely many points of $X(\overline{\Q }_p )\cap J(\Q _p ^{\mathrm{nr}})_{\divv }$ lie in $X(\Q _p ^{\mathrm{nr}})$ (in fact, he proves a considerably more general statement \cite[Proposition 8.2.1]{Ray1}). One possible novelty of our approach is that we make minimal use of the Galois action on torsion points of $J$. We hope that this will make the results amenable to generalisation to other cases of the Zilber--Pink conjecture.

After Faltings' proof of Mordell--Lang it was shown in several works how to bound the size of the set in terms of invariants of the curve. The state of the art is the work of Dimitrov--Gao--Habegger \cite{DGH} and K\"uhne \cite{kuhne2021equidistribution}, and separately Yuan \cite{yuan2021arithmetic}, who show that there is a constant $c(g)$ such that
\[
\# X(\overline{\Q }_p )\cap \Gamma <c(g)^{r+1}
\]
(with no condition on the reduction at $p$ or the field of definition of $X$). However the constant $c(g)$ seems hard to make explicit.

In the case $r=0$, part (3) of the theorem reduces to the Buium--Coleman's explicit Manin--Mumford bound \cite{buium:96} \cite{coleman:ramified}. In fact, the exponent $p^{3g}$ is a minor improvement on that of Buium which was obtained by Poonen \cite{poonen}. On the other hand, if we considered the intersection of $X(L)$ with $\Gamma $ for $L$ a finite unramified extension of $K$, we have Coleman's effective Chabauty bound \cite{coleman:chabauty}
\[
\# X(L)\cap \Gamma \leq \# X(\kappa _L )+2g-2,
\]
where $\kappa _L$ denotes the residue field of $L$. Part (2) of the theorem is a generalisation of Coleman's result \cite{coleman:ramified} that when $p\geq 2g$, torsion points on $X$ (i.e. the divisible hull of $\{ 0\}$) are unramified. The main idea of this paper (and the proof of part (3) of the theorem) is that one can combine these three ideas (Buium's bound on unramified torsion, Coleman's bound on ramified torsion, and Coleman's effective Chabauty bound) to get an explicit Mordell--Lang bound.

\subsection*{Acknowledgements}
We would like to thank Gabriel Dill for helpful feedback on an earlier version, and Arnab Saha for helpful conversations. This collaboration was initiated at ``Rational points on modular curves'' in ICTS Bangalore. We would like to thank the organisers for the excellent working conditions. This research was supported by Royal Society research fellowship URF$\backslash $R1$\backslash $201215 and grant RF$\backslash $ERE$\backslash $231161.
\section{Proof of part (1): the Buium step}
\begin{definition}[$p$-derivation] Let $u: A\rightarrow B$ be a ring homomorphism. A $p$-derivation of $u$ is a set-theoretic map $\delta:A \rightarrow B$ that satisfies  for 
all $x,y \in A$,

$(i)~ \delta (1) = 0$
\color{black}

$(ii)~ \delta (x+y) = \delta x + \delta y + C_p(u(x),u(y)) $

$(iii)~ \delta(xy) = u(x)^p \delta y + u(y)^p \delta x + p \delta x \delta y$\\
where
$$C_p(X,Y) =  \frac{X^p+Y^p -(X+Y)^p}{p}\in\mathbb{Z}[X,Y].$$ 
\end{definition}

\noindent Given a $p$-derivation, one can associate a unique lift of Frobenius (of $u$) $\phi: A\rightarrow B$ defined by $$\phi(x):=u(x)^{p}+p\delta(x).$$
Moreover, it is not difficult to see that there is a one to one correspondence between the $p$-derivations of $u$ and the ring homomorphisms $u_1:A\rightarrow W_{1}(B)$ such that $p_{1} \circ u_1=u$ given by $u_{1}(x)=(u(x),\delta(x))$, where $W_{1}(B)$ denote the $p$-typical  Witt vectors of length $2$ and $p_1: W_1(B)\rightarrow B$ be the first projection.\medskip

\noindent Let $R=\Z_{p}^{\mathrm{nr}}$, the ring of integers of $\Q _p ^{\mathrm{nr}}$. The unique lift of Frobenius $\sigma :R\rightarrow R$ gives a $p$-derivation $\delta$ of the identity map defined by $$\delta(x):=\frac{\sigma(x)-x^p}{p}.$$
Given a smooth scheme $Y/R$, let $Y_{1}=Y\times R/p^2$ be the reduction of $Y$ modulo $p^2$. Following section $1.4$ in \cite{buium:96}, one can define the first arithmetic jet space of $Y_{1}$ denoted as $Y^1_{0}$ (also called the first $p$-jet space).
By Bertapelle--Previato--Saha \cite{BPS}, one can reinterpret Buium's arithmetic jet spaces using the functor of points as follows:
\begin{lemma}[Functor of points of Jet spaces] Let $Y/R$ be a  scheme. The first arithmetic jet space represents the functor from the category of $R$-Algebras to Sets given by 
$$Y^1(B)=Y(W_1(B)) ~\mathrm{for ~any} ~R\mathrm{-algebra}~ B.$$
\end{lemma}
\noindent Hence $Y^1$ is an $R$-scheme  with the projection map $u:Y^1\longrightarrow Y$. Then $Y^1_{0}$ in Section $1.4$ in \cite{buium:96} is indeed the reduction of $Y^1$ modulo $p$. Moreover, if $Y/R$ is a group scheme then $Y^1_{0}$ is a $k$-group scheme.
Since $R$ is equipped with a $p$-derivation, it corresponds to an $R$-algebra map $R\rightarrow W_{1}(R)$. Thus we have a functorial map $Y(R)\xrightarrow{\nabla^1} Y(W_1(R)).$ The functor of points of jet spaces then gives a map $$\nab: Y(R)\rightarrow Y_{0}^{1}(k)$$ by the following compositions $\nab: Y(R)\xrightarrow{\nabla^1} Y^1(R)\xrightarrow{\mod p} Y_{0}^{1}(k)$ that satisfies

\begin{align}\label{mod-p-com}
\xymatrix{
Y(R)\ar[rd]_{\mod p} \ar[rr]^{\nab}& &  Y^1_{0}(k)\ar[ld]^{\overline{u}}\\
& Y_{0}(k)          }
\end{align}
For the explicit bound in Theorem \ref{thm:main1} we need an explicit bound on $\Gamma /p\Gamma $. Finiteness of $\Gamma /p\Gamma $ is proved in \cite[Lemma 9.1.2]{Ray1} and it is straightforward to make this quantitative.
\begin{lemma} \label{bd-Gamma-tor} Let $\Gamma$ be a subgroup of $(\mathbb{Q}/\mathbb{Z})^n$ for some positive integer $n$. Let $\Gamma[p]$ be the $p$-torsion subgroup of $\Gamma$. Then $$ \#(\Gamma/p\Gamma) \leq  \#(\Gamma [p])\leq p^n.$$
\end{lemma}
\begin{proof} 
Let $\Gamma=\Gamma' \oplus \Gamma_p$, where $\Gamma'$ is the prime to $p$ torsion subgroup and $\Gamma_{p}$ is the $p$-power torsion subgroup of $\Gamma$. Then we have  $\Gamma/p\Gamma=\Gamma_{p}/p\Gamma_{p}$. 

We have 
\[
\Gamma _p \simeq (\Q _p /\Z _p )^a \oplus (\oplus _{i=1}^b \Z /p^{n_i }\Z )
\]
for some $n_i \geq 1$ and $a+b\leq n$. Then $\# \Gamma [p]=p^{a+b}$ and $\# \Gamma _p /p\Gamma _p =p^b$.
\end{proof}
We deduce the following lemma, which also follows immediately from the proof of \cite[Lemma 9.1.2]{Ray1}.
\begin{lemma} \label{bd-Gamma} Let $\Gamma \leq J(R)$ be a subgroup with $\dim \Gamma \otimes _{\Z }\Q =r$, and $\Gamma[p]$ be the $p$-torsion subgroup of $\Gamma$. Then $$ \#(\Gamma/p\Gamma) \leq p^r \cdot \#(\Gamma [p]) \leq p^{g+r}.$$
\end{lemma}
\begin{proof} Let $\Gamma_{\tor}$ be the torsion subgroup of $\Gamma$. Then $\dim _{\F _p }\Gamma _{\tor }\leq g$ by Lemma \ref{bd-Gamma-tor}. Let $\Gamma ' =\Gamma /\Gamma _{\tor}$. Then it is enough to show that $\dim _{\F _p }\Gamma ' \otimes \F _p \leq \dim _{\Q }\Gamma ' \otimes \Q $. This follows from the fact that if $x_i \in \Gamma '$ are linearly dependent in $\Gamma ' \otimes \Q $, then they are linearly dependent in $\Gamma ' \otimes \F _p $.
\end{proof}
We recall the following results from Buium's work on the Manin--Mumford conjecture.
\begin{theorem}[Buium  \cite{buium:96}, Proposition $1.10$]\label{thm:buium1} Let $X/R$ be a smooth proper curve of genus $g>1$. Then $X^1_0$ is affine.
\end{theorem}

\begin{theorem}\label{Bui-thm} Let $X/R$ be a curve of genus $g>1$. Let $J$ be the Jacobian of $X$ and $\Gamma\leq J(R)$ be a subgroup such that
\[
\dim _{\Q _p }\Gamma \otimes _{\Z }\Q \leq r.
\]
Let $P$ be a point on $X(R)$ and $\iota_{P}: X\longrightarrow J$ be the Abel-Jacobi map with respect to $P$. Then 
$$\#\nab(\iota_{P}(X(R))\cap \Gamma)\leq p^{3g+r}3^{g}[p(2g-2)+6g]g!.$$
\end{theorem}

\begin{proof} Let $B=p J^1_0$ be the maximal abelian subvariety of $J^1_0$. Note that $\nab(\Gamma)$ has finite image in $J_{0}^1(k)/B$. Indeed, the image set is bounded by $\#(\Gamma/p\Gamma)$. Then $\#(\Gamma/p\Gamma)\leq p^{g+r}$ by Lemma \ref{bd-Gamma}. Therefore we have $$\nab(X(R) \cap \Gamma )\subset \bigcup_{i=1}^{N} (P_{i}+B)\cap X^1_{0}$$
for some $P_{i}\in J^1_{0}(k)$ and $N\leq p^{g+r}$. Then by Theorem \ref{thm:buium1} we know that $X^1_{0}$ is affine, while $B$ is proper and they are closed subvarieties of $J^{1}_{0}$. Therefore $B\cap X^1_{0}$ is finite. Moreover, Buium proved \cite[p.357]{buium:96} that $$\#(B\cap X^1_{0})\leq p^{2g}3^{g}[p(2g-2)+6g]g!$$
This implies that $$\#\nab(X(R) \cap \Gamma)\leq p^{3g+r}3^{g}[p(2g-2)+6g]g!.$$
\end{proof}
As a corollary we deduce part (1) of Theorem \ref{thm:main1}.
\begin{corollary}\label{cor:cor1}
With the hypotheses as in Theorem \ref{Bui-thm}, and let $\red (\iota_{P}(X(R))\cap \Gamma )$ denote the mod $p$ reduction of the set $\iota_{P}(X(R))\cap \Gamma$. Then 
$$\# \red((\iota_{P}(X(R))\cap \Gamma)) \leq p^{3g+r}3^{g}[p(2g-2)+6g]g!.$$
\end{corollary}
\begin{proof} Follows from Theorem \ref{Bui-thm} and the commutative diagram \eqref{mod-p-com}.
\end{proof}

\section{Proof of part (2): the Coleman step}
For $X/R$ and $\overline{z}$ in $X(\overline{\F }_p )$, we will denote by $]\overline{z}[$ the tube of $\overline{z}$ in the sense of Berthelot \cite{LS}. Recall that this is a rigid analytic space whose $\mathbb{C}_p $-points are the preimage of $\overline{z}$ under the reduction map. Fix $\overline{z}\in X(\overline{\F}_p )$ and $z$ in $]\overline{z}[(\overline{\Q }_p )$. In this section we prove the following proposition.
\begin{proposition}\label{prop:coleman1}
Let $z\in ]\overline{z}[(\overline{\Q }_p )$. If $\val (\int ^{z}_{z_0 }\omega ) \in \Z$ for all $z_0 \in X(\Q _p ^{\mathrm{nr}})$ and all $\omega \in H^0 (X_{\Q _p ^{\mathrm{nr}}},\Omega )$, then $z\in X(\Q _p ^{\mathrm{nr}})$.
\end{proposition}
We first note that this implies part (2) of Theorem \ref{thm:main1}. Indeed, if $P\in X(\overline{\Q }_p )\cap J(\Q _p ^{\mathrm{nr}})_{\divv}$, say $n\cdot P\in J(\Q _p ^{\mathrm{nr}})$, and if $\omega _J$ denotes the differential in $H^0 (J_{\Q _p ^{\mathrm{nr}}},\Omega )$ corresponding to $\omega \in H^0 (X_{\Q _p ^{\mathrm{nr}}},\Omega )$, then 
\[
\int ^z _{z_0 }\omega =\frac{1}{n}\int ^{nP}\omega _J -\int ^{z_0 }_b \omega \in \Q _p ^{\mathrm{nr}}.
\]

The proof of Proposition \ref{prop:coleman1} will be a simple consequence of what we call the \textit{Coleman expansion} of an abelian integral. This is a remarkable method, devised by Coleman in \cite{coleman:ramified}, to decompose the Taylor expansion of an abelian integral into pieces of a rather algebraic nature. From this one obtains a description of the slopes of the Taylor expansion in terms of the action of powers of Verschiebung on the class of the differential in cohomology.
\subsection{The Coleman expansion of an abelian integral}
Let $U$ be an affine open neighbourhood of $x\in X(\overline{\F }_p )$ lifting to a formal affine $\mathfrak{U}\subset \mathfrak{X}$, where $\mathfrak{X}$ is the completion of $X$ along the special fibre. We will denote by $H^0 (\mathfrak{U},\Omega )$ the module of formal differentials on $U$. Following Coleman \cite{coleman:ramified}, we say $\eta \in H^0 (\mathfrak{U},\Omega )$ is of the second kind if there is an open cover of $\mathfrak{X}$ containing $\mathfrak{U}$ such that $(\mathfrak{U},\eta )$ extends to a Cech cocycle $(\mathfrak{U}_{\alpha },\omega _{\alpha })$ for the formal de Rham complex (i.e. so that $\omega _{\alpha }-\omega _{\beta }\in d\mathcal{O}(\mathfrak{U}_{\alpha }\cap \mathfrak{U}_{\beta })$. We have an isomorphism
\[
H^1 _{\dR}(X/R) \simeq \colim H(\mathcal{C}((\mathfrak{U}_{\alpha }))
\]
where the limit is over formal affine open covers $(\mathfrak{U}_{\alpha })$, and $H(\mathcal{C}((\mathfrak{U}_{\alpha })))$ is the group of Cech cocycles modulo exact cocycles. Given a differential of the second kind $\eta $, we will write $[\eta ]$ for the corresponding class in $H^1 _{\dR}(X/R)$. 

As before let $\sigma $ be the Frobenius lift $R\to R$, for any (formal) $R$-scheme $Y$ let $Y^\sigma $ denote pullback of $Y$ along $\sigma $. For a function $f$ on $Y$ we denote by $f^{\sigma }$ the pullback along the map $Y^\sigma \to Y$, and similarly for differentials. We have an $R$-linear Frobenius morphism
\[
X_{\overline{\F }_p }\to X_{\overline{\F }_p }^\sigma .
\]
In what follows, $\phi $ will denote a lift of Frobenius to a morphism
\[
\mathfrak{U}\to \mathfrak{U}^{\sigma },
\]
where $\sigma \in \Gal (\Q _p ^{\nr}|\Q _p )$ is the lift of Frobenius, and $\mathfrak{U}^{\sigma }$ is the pullback of $\mathfrak{U}$ along $\sigma $. $\phi $ determines a map 
\[
F:H^1 _{\dR}(X/R)\to H^1 _{\dR}(X/R)
\]
given, on differentials $\omega $, by $\omega \mapsto \phi ^* \omega ^{\sigma }$.
By Coleman (\cite[Corollary 9.1]{coleman:ramified}) given $\omega ,\eta \in H^0 (\mathfrak{U},\Omega )$ of the second kind such that $V[\omega ]=\eta $, we have 
\[
\omega =\frac{\phi ^* \eta ^{\sigma }}{p}+df
\]
where $f\in \mathcal{O}(\mathfrak{U})$. In particular $f$ is integral. Given $\omega $ we obtain a sequence $\omega _i$ of differentials and $f_i$ of functions on $\mathfrak{U}$ defined by $\omega _0 =\omega $ and 
\[
\omega _i =\frac{\phi ^* \omega ^{\sigma } _{i+1}}{p}+df_i .
\]

Given a nonzero cohomology class $\eta \in H^1 _{\dR}(X/R)$, we define its valuation $\val (\eta )$ to be the largest non-negative integer $n$ such that $\eta $ lies in $p^n H^1 _{\dR}(X/R)$. If $\eta /p^{\val (\eta )}$ lies in $H^0 (X,\Omega )+pH^1 _{\dR}(X/R)$, then we define 
\[
\overline{\eta }\in H^0 (X_{\overline{\F }_p },\Omega )
\]
to be a global differential with image in $H^1 _{\dR}(X_{\overline{\F }_p }/\overline{\F }_p )$ equal to that of the reduction modulo $p$ of $\eta /p^{\val (\eta )}$. 

The motivation for such a definition is the $F$-crystal structure on $H^1 _{\dR}(X/R)$: the $R$-module $H^1 _{\dR}(X/R)$ admits a $\sigma ^{-1}$-semilinear map
\[
V:H^1 _{\dR}(X/R)\to H^1 _{\dR}(X/R)
\]
satisfying
\begin{align*}
& VF=FV=p \\
& V(H^1 _{\dR}(X/R ))=H^0 (X,\Omega )+pH^1 _{\dR}(X/R).
\end{align*}
From this we deduce the following.
\begin{lemma}[Coleman, \cite{coleman:ramified}]\label{coleman:ramified}
For all $\eta \in H^1 _{\dR}(X/R)$, 
\[
\val (\eta )\leq \val (V(\eta ))\leq \val (\eta )+1.
\]
Furthermore $\val (V(\eta ))=\val (\eta )$ if and only if
 \[
 V(\eta )\in p^{\val (\eta )}\cdot H^0 (X,\Omega )+p^{\val (\eta )+1}\cdot H^1 _{\dR}(X/R).
 \]
\end{lemma}
We now recall the definition of the Coleman expansion. Let $\omega \in H^0 (X,\Omega )$ be a differential which is nonzero modulo $p$, and $\overline{z}\in X(\overline{\F }_p )$. Define a sequence $n_i =n_i (\omega )$ and $k_i =k_i (\omega ,\overline{z})$ of non-negative integers for $i\geq 0$. $n_0=0$ and $k_0 $ is the order of $\omega $ at $\overline{z}$. For $i>0$, $n_i$ is defined to be the $i$th positive integer $m$ such that $\val (V^m ([\omega ] ))=\val (V^{m-1}([\omega ] ))$. Define $k_i (\omega ,\overline{z}):=\ord _{\overline{z}}(\overline{V^m (\omega  )})+1$. Coleman proved \cite[Lemma 6]{coleman:ramified} that the sequences $(n_i)$ and $(k_i )$ are infinite.

\begin{lemma}
We have $\val ([\omega _i ])=p^{i-j}$, where $j$ is the largest non-negative integer such that $n_j \leq i$.
\end{lemma}
\begin{proof}
By definition, $\val ([\omega _{\ell }])=\val ([\omega _{\ell -1}])$ if $\ell $ is in the sequence $(n_j)$. By Lemma \ref{coleman:ramified}, $\val ([\omega _\ell ])$ is $\val ([\omega _{\ell-1}])+1$ if $\ell$ is not in the sequence $(n_{j})$. Hence 
\[
\val (\omega _i ) =\# \{ 1,\ldots ,i \} - \# \{ n_1 ,\ldots ,n_j \}. 
\]
\end{proof}
Define $\nu _i :=\omega _{n_i }/p^{n_i -i}$. Note that, by Lemma \ref{coleman:ramified}, $\nu _i $ lies in $H^0 (X,\Omega )$ modulo $p$. 

Finally, suppose $T$ is an integral parameter at $z_0 $, and $\phi$ is a lift of Frobenius sending $T^{\sigma }$ to $T^p$ (such a lift exists, see e.g. \cite{coleman:torsion}). 
\begin{lemma}[Coleman, \cite{coleman:ramified}, p.624]
There are functions $g_i (T)\in T\cdot R[\! [T]\! ]$ such that 
\[
\nu _i =\frac{\phi ^{(n_{i+1}-n_i )*}\nu _{i+1}^{\sigma ^{n_{i+1}-n_i}}}{p}+dg_i .
\]
\end{lemma}
\begin{proof}
For the sake of completeness we recall some details of the proof. We have
\begin{align*}
[\nu _{i+1}] & =\frac{[\omega _{n_{i+1}}]}{p^{n_{i+1}-i-1}} \\
 & =  \frac{V^{n_{i+1}-n_i}([\omega _{n_i} ])}{p^{n_{i+1}-i-1}} \\
\end{align*}
hence
\begin{align*}
F^{n_{i+1}-n_i }([\nu _{i+1}]) & = \frac{p^{n_{i+1}-n_i}([\omega _{n_i} ])}{p^{n_{i+1}-i-1}}  \\
& =\frac{[\omega _{n_i} ]}{p^{n_i -i-1}}=p[\nu _i ].
\end{align*}
in $H^1 _{\dR}(X/R)$, hence there is a rigid analytic function $g_i$ on $]U[$ such that
\[
\nu _i =\frac{\phi ^{(n_{i+1}-n_i )*}\nu _{i+1}^{\sigma ^{n_{i+1}-n_i}}}{p}+dg_i .
\]
Coleman shows \cite[Corollary 9.1]{coleman:ramified} that in fact $g_i $ can be taken to be integral, i.e. in $\mathcal{O}(\mathfrak{U})$. Taking the formal expansion at $z_0$ gives the power series in the statement of the Lemma.
\end{proof}
We define the \textit{Coleman expansion} of $\int ^z _b \omega $ to be the series expansion
\[
\int ^z _b \omega =\int ^{z_0 }_b \omega +\sum _{i>0}\frac{g_i (z^{p^{n_i }})}{p^i }.
\]
Note that $g_i (z^{p^{n_i }})\in T^{p^{n_i }}\cdot R[\! [T^{p^{n_i }}]\! ]$.

Recall that the \textit{Newton polygon} of a power series $F(T)=\sum _{n\geq 0}a_n T^n \in \Q _p ^{\mathrm{nr}}[\! [T]\! ]$ is the highest convex polytope in $\mathbb{R}^2$ lying below all the points $(n,\val (a_n ))$. We say $\lambda >0$ is a negative slope of $F$ if there is a line segment of the Newton polygon of $F$ of gradient $-\lambda $. We recall the following properties. 
\begin{lemma}\label{lemma:slope_pictures}
\begin{enumerate}
\item There is a point $z$ with $\val (z)=\lambda $ and $F(z)=0$ if and only if $\lambda $ is a negative slope of $F$. 
\item If $\val (z)=\lambda $, and $\lambda $ is not a negative slope of $F(T)=\sum a_n T^n$, and there are consecutive slopes $\mu _1 $ and $\mu _2$, joined at a vertex $(n,m)$, with
\[
-\mu _1 <\lambda <-\mu _2 ,
\]
then $\val (F(z))=\lambda n+m$.
\end{enumerate}
\end{lemma}
\begin{proof}
The first statement is well known (see e.g. \cite[IV.4]{koblitz}). The second statement is easier: $n\lambda +m=\val (a_n z^n )$ is the unique minimum of $\{ a_k z^k :k\geq 0 \}$. Indeed if $\val (a_k z^k )\leq n\lambda +m$ then $(k,\val (a_k ))$ lies on or below the line $y+\lambda x=n\lambda +m$. However our assumptions mean that the Newton polygon is above this line, and $(n,m)$ is the unique point meeting the Newton polygon.
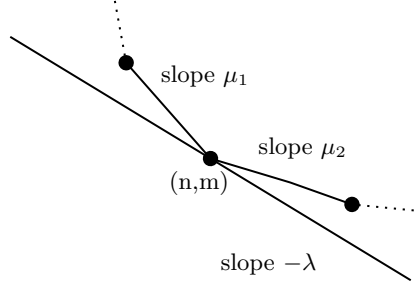
\begin{center}

\tikzset{every picture/.style={line width=0.75pt}} 

\begin{tikzpicture}[x=0.75pt,y=0.75pt,yscale=-1,xscale=1]

\draw  [dash pattern={on 0.84pt off 2.51pt}]  (148.33,56) -- (154,87) ;
\draw [shift={(154,87)}, rotate = 79.64] [color={rgb, 255:red, 0; green, 0; blue, 0 }  ][fill={rgb, 255:red, 0; green, 0; blue, 0 }  ][line width=0.75]      (0, 0) circle [x radius= 3.35, y radius= 3.35]   ;
\draw    (154,87) -- (196.33,135) ;
\draw [shift={(196.33,135)}, rotate = 48.59] [color={rgb, 255:red, 0; green, 0; blue, 0 }  ][fill={rgb, 255:red, 0; green, 0; blue, 0 }  ][line width=0.75]      (0, 0) circle [x radius= 3.35, y radius= 3.35]   ;
\draw    (196.33,135) -- (236.33,147) -- (267.33,158) ;
\draw  [dash pattern={on 0.84pt off 2.51pt}]  (298.33,161) -- (267.33,158) ;
\draw [shift={(267.33,158)}, rotate = 185.53] [color={rgb, 255:red, 0; green, 0; blue, 0 }  ][fill={rgb, 255:red, 0; green, 0; blue, 0 }  ][line width=0.75]      (0, 0) circle [x radius= 3.35, y radius= 3.35]   ;
\draw    (95.83,74) -- (296.83,196) ;

\draw (174,140) node [anchor=north west][inner sep=0.75pt]   [align=left] {{\small (n,m)}};
\draw (200,181) node [anchor=north west][inner sep=0.75pt]   [align=left] {{\small slope $\displaystyle -\lambda $}};
\draw (170,87) node [anchor=north west][inner sep=0.75pt]   [align=left] {{\small slope $\displaystyle \mu _{1}$}};
\draw (219,123) node [anchor=north west][inner sep=0.75pt]   [align=left] {{\small slope $\displaystyle \mu _{2}$}};

\end{tikzpicture}
\end{center}
\end{proof}
The following simple but fundamental result due to Coleman allows us to describe the slopes of $\int \omega $ purely in terms of the Coleman expansion.
\begin{lemma}\label{lemma:slop}
Suppose $\lambda $ is a (negative) slope of the Newton polygon of $\int _{z_0} \omega \in \Q _p ^{\nr }[\! [T]\! ]$. If $\lambda < \frac{1}{2g-2}$, then
\begin{equation}\label{eqn:slopes}
\lambda =\frac{1}{p^{n_{i+1} (\omega )}k_{i+1}(\omega ,\overline{z})-p^{n_i (\omega )}k_{i}(\omega ,\overline{z})}
\end{equation}
for some $i\geq 0$.
\end{lemma}
\begin{proof}
Write $\int _{z_0 }\omega =\sum a_n T^n $. The Coleman expansion implies that if $\val (a_m )\leq -i$, then $m\in p^{n_i }\cdot \Z_{\geq k_i }$. Moreover Riemann--Roch implies $k_i \leq 2g-1$, which (together with the fact that $p>2g-2$) implies that the first coefficient of the $T$-adic expansion of $\int _{z_0} \omega $ of valuation $-i$ is at $k_i p^{n_i }$. Again using the fact that $p>2g-2$, we deduce that the slopes to the left of the vertex $(k_0 ,0)$ are exactly of the form \eqref{eqn:slopes}. Since $k_0 \geq 2g-2$, the slopes to the left of $(k_0 ,0)$ must be at least $\frac{1}{2g-2}$.
\end{proof}
\subsection{Proof of Theorem \ref{thm:main1} part (2)}
Let $z$ be as in the statement of Proposition \ref{prop:coleman1}. To prove that $z$ is in $X(\Q _p ^{\mathrm{nr}})$, it will be enough to prove that, for all $z_0 \in ]\overline{z}[(\Q _p ^{\mathrm{nr}})$ and all integral parameters $T$ at $z_0$, $\val (T(z))\in \Z $. Indeed, if $z_0$ is any such point, $T$ is an integral parameter, and $z$ is a ramified point congruent to it, then we may take the least $n$ such that the image of $T(z)$ in $\overline{\Z }_p /p^n \overline{\Z }_p $ does not lie in $R/p^n R$. Let $z_0 ' \in ]\overline{z}[(\Q _p ^{\mathrm{nr}})$ be a point such that $T(z_0 ')-T(z)\in p^{n-1}\overline{\Z} _p$.

Fix such a $z_0 $ and such a $T\in \mathcal{O}_{X,\overline{z}}$. Define
\[
F_\omega (T):=\int _{z_0 } \omega |_{]\overline{z}[}\in \Q _p ^{\mathrm{nr}}[\! [T]\! ].
\]
Let $(n_i (\omega ),k_i (\omega ,\overline{z}))$ be the Coleman sequence of $F_{\omega }$. Let $\lambda $ be the valuation of $T(z)$. 
Recall that if $\lambda <1/(2g-2)$, then slopes adjacent to $\lambda $ are of the form
\[
\frac{1}{p^{n_i }k_i -p^{n_{i-1}}k_{i-1}}>\lambda >\frac{1}{p^{n_{i+1} }k_{i+1} -p^{n_{i}}k_{i}}
\]
and the endpoint is $(p^{n_i}k_i ,-i)$. Hence by Lemma \ref{lemma:slope_pictures} the valuation of $F_{\omega }(x)$ is $\lambda p^{n_i }k_i -i$, and hence satisfies
\[
\frac{p^{n_i }k_i }{p^{n_i }k_i -p^{n_{i-1}}k_{i-1}}>\val (F_{\omega }(x))+i >\frac{p^{n_i}k_i }{p^{n_{i+1} }k_{i+1} -p^{n_{i}}k_{i}}.
\]
Our assumptions on $p$ and $g$ imply that the only way this can be integral is if $\lambda p^{n_i }k_i =1$. By Lemma \ref{lemma:slop}, it follows that if $\lambda <\frac{1}{2g-2}$ and $\int ^z _b \omega \in \Q _p ^{\mathrm{nr}}$, then it must be the case that for all $\omega \in H^0 (X,\Omega )-pH^0 (X,\Omega )$, $\lambda $ is equal to either $\frac{1}{p^{n_i }k_i }$ or $\frac{1}{p^{n_i }k_i -p^{n_{i-1}}k_{i-1}}$ for some $i$ (where $n_i$ is short for $n_i (\omega )$, and $k_i$ is short for $k_i (\omega ,\overline{z})$). Since the $k_i$ are all integers between $1$ and $2g-2$, and $p>2g$, this is only possible if there is some $k$ and $n$ such that for all $\omega $, for some $i$, $k_i (\omega ,\overline{z})=k$ and $n_i (\omega )=n$. Note that in particular this means that for all $\omega $, $\ord _{\overline{z}}(\overline{V^n (\omega )})=k-1$. 

However this is impossible.  Indeed, let $\omega _1 $ and $\omega _2 $ be two global differentials which are linearly independent modulo $p$, suppose $n=n_i (\omega _1 )=n_j (\omega _1 )$, $k=k_i (\omega _1 ,\overline{z})=k_j (\omega _2 ,\overline{z})$ and $i\leq j$. Suppose furthermore that the valuation of $V^n (\omega _2 )$ is maximal among all differentials in $H^0 (X,\Omega )-pH^0 (X,\Omega )$. Choose $\alpha \in R^{\times }$ such that 
\[
\ord _{\overline{z}}(\overline{V^n ([\omega _1 ])}-\alpha \overline{V^n ([\omega _2 ])})>k-1
\]
(such an $\alpha$ exists since both differentials have order equal to $k$ at $\overline{z}$). It follows that either 
\[
\ord_{\overline{z}}(\overline{V^n (p^{j-i}\omega _1 -\alpha \omega _2 )})>k-1
\] 
or 
\[
\val (V^n ([p^{j-i}\omega _1 -\alpha \omega _2 ]))>\val (V^n ([\omega _2 ])),
\]
either of which results in a contradiction.

If $\lambda <\frac{1}{2g-2}$ and $\lambda $ is a negative slope of some $\int \omega $, then $\lambda $ is of the form $\frac{1}{p^{n_i (\omega )}k_i (\omega ,\overline{z})-p^{n_{i-1}(\omega )}k_{i-1} (\omega ,\overline{z})}$ for some $i$ and $\omega $. By Riemann--Roch, the $k_i$ are less than $p$, and hence there are \textit{unique} non-negative integers $a,b $ with $a<b$, and $k,\ell $ with $0\leq k,\ell \leq 2g-2$, such that
\[
\lambda = \frac{1}{kp^b-\ell p^a }
\]
If $\eta $ is any global one-form, and $\val (\int ^z _{z_0 }\eta )\in \Z $, then either $\lambda $ is a negative slope of $\int _{z_0 }\eta $, or there is an integer $j$ such that 
\[
\frac{p^{n_j (\eta )}k_j (\eta ,\overline{z})}{p^{n_i (\omega )}k_i (\omega ,\overline{z})-p^{n_{i-1}(\omega )}k_{i-1} (\omega ,\overline{z})}\in \Z  ,
\]
which is impossible. However, arguing as above, if $\eta $ is any $1$-form linearly independent of $\omega $, then there is a linear combination $\eta ' =\alpha \omega +\beta \eta $ such that there is no $j$ with $n_j (\eta ' )=n_i (\omega )$ and $k_j (\eta ' ,\overline{z})=k_j (\omega ,\overline{z})$.

Finally suppose that $\lambda \geq 1/{2g}$. Let $\omega \in H^0 (X_R ,\Omega )$ be a differential which does not vanish modulo $p$ at $\overline{z}$. Then $\val (F_{\omega }(z))=\lambda $, which by assumption is not in $\Z $, completing the proof.

\section{Proof of part (3): The Chabauty--Coleman step}
\begin{proof}[Proof of Theorem \ref{thm:main1} part (3)]
Assume we are in the setting of part (3) of Theorem \ref{thm:main1}. Let $S \subset X(\overline{\F }_p )$ denote the set of points whose residue disk contains a point lying on $\Gamma $. By part (2), $S$ contains a point of $\Gamma $ if and only if $S$ contains an unramified point of $\Gamma $, hence
\[
\# S\leq p^{3g+r}3^g [p(2g-2)+6g]g!.
\]
To prove part (3) of Theorem \ref{thm:main1}, it will hence be enough to prove
\[
N(S):=\sum _{s \in S}\left( \# X(\Q _p ^{\mathrm{nr}})\cap \Gamma \cap ]s[(\Q _p ^{\mathrm{nr}})-1 \right) \leq 2r.
\]
Since $r=\dim \Gamma \otimes _{\Z }\Q $ is less than $g$, we have a $g-r$ dimensional space $\Gamma ^{\perp}$ of vanishing differentials $\omega $, i.e.
\[
\Gamma ^{\perp}:=\{ \omega \in H^0 (X_{\Z _p ^{\mathrm{nr}}},\Omega ):\forall P\in \Gamma ,\int ^P \omega =0 \}.
\]
Here we view $\omega $ as a differential on $J$ via pullback along the Abel--Jacobi morphism. For $s\in S$, define
\[
n(s):=\min \{\ord _s \overline{\omega }:\omega \in \Gamma ^{\perp}-p\cdot \Gamma ^{\perp} \}
\]
By Coleman \cite{coleman:chabauty}, we have 
\[
\# X(\Q _p ^{\mathrm{nr}})\cap \Gamma \cap ]s[(\Q _p ^{\mathrm{nr}})-1 \leq n(s).
\]
By Stoll \cite[Theorem 6.4]{independence}, we have
\[
\sum _{s\in S}n(s)\leq 2r
\]
hence $N(S)\leq 2r$, completing the proof of the Theorem.
\end{proof}

\section{Mordell implies Mordell--Lang}
In this section we give the proof of Corollary \ref{cor:1}. Let $X$ be a curve over a number field $K$, and $\Gamma \subset J(\overline{K})$ the divisible hull of the subgroup generated by the $K$-points of $J$. By part (2) of Theorem \ref{thm:main1}, and the Hermite--Minkowski theorem, it is sufficient to establish the following.
\begin{proposition}\label{prop:2}
There is an integer $N>0$ such that every point of $X\cap \Gamma $ is defined over an extension of $K$ of degree less than $N$.
\end{proposition}
What we in fact prove is the following.
\begin{proposition}\label{prop:3}
There is a finite extension $K/\Q $ and an integer $n$ such that 
\[
S_n :=\{ (\sigma _1 (x),\ldots ,\sigma _n (x)):x\in X\cap \Gamma ,\sigma \in \Gal (K) \}
\]
is not Zariski dense in $X^n $.
\end{proposition}
\begin{lemma}
Proposition \ref{prop:3} implies Proposition \ref{prop:2}.
\end{lemma}
\begin{proof}
Let $Z$ be any proper closed subvariety of $X^n$ with the property that,  for all $\sigma _1 ,\ldots ,\sigma _n \in \Gal (K)$, if $
(z_1 ,\ldots ,z_n )$ is in $Z(\overline{K})$ then $(\sigma _1 (z_1 ),\ldots ,\sigma _n (z_n ))$ is in $Z(\overline{K})$. 
We show by induction on $n$ that the projection of $S_n \cap Z$ to $X$ (by the first coordinate) consists of points of degree which can be bounded in terms of $Z$. 

The case $n=1$ is obvious. Let $\pi :Z\to X^{n-1}$ be projection onto the first $n-1$ coordinates. If $\pi $ is not generically finite, then $\dim (\pi (Z))<n-1$, and we may apply our inductive hypothesis, to $\pi (Z\cap S_n )=\pi (Z)\cap (S_n )$ to deduce the result. 

Hence we may assume that the projection map $\pi :Z\to X^{n-1}$ is generically finite, say finite outside a closed subset $Z' \subset X^{n-1}$. Since $Z$ and $\pi $ are defined over $K$, $Z'$ also has the property that, for all $\sigma _i \in \Gal (K)$, $(z_i )\in Z' (\overline{K})$ implies $(\sigma _i (z_i ))\in Z' (\overline{K})$. Hence by the argument above the projection of $\pi ^{-1}(Z')\cap S_n $ to $X$ consists of points of bounded degree. Let $U'$ denote the complement of $Z'$ in $\pi (Z)$, and let $U$ denote $\pi ^{-1}(U')$. By the argument above, the projection of $\pi ^{-1}(Z')\cap S_n$ to $X$ consists of points of bounded degree, hence it is enough to show that the projection of $U\cap S_n $ to $X$ consists of points of bounded degree.

Let $M$ denote the degree of the finite map $\pi :U\to U'$. Let $(\sigma _1 (x),\ldots ,\sigma _n (x))$ be a point in $U\cap S_n$. If $(\sigma _1 (x),\ldots ,\sigma _n (x))$ is in $U$ for some $\sigma _i \in \Gal (K)$, then $(\sigma _1 (z),\ldots ,\sigma (z))$ is in $U$ for \textit{all} $\sigma $ in $\Gal (K)$, i.e.
\[
[K(z):K]\leq M.
\]
\end{proof}
Our next act of subterfuge is reducing to the following proposition.
\begin{proposition}\label{prop:4}
There is an integer $m$ such that $(S_n ) ^m \subset X^{nm}({\overline{K}})$ is not Zariski dense in $X^{nm}_{\overline{K}}$.
\end{proposition}
Note that Proposition \ref{prop:4} implies Proposition \ref{prop:3}, since the closure of a product is the product of the closures.

We now deduce Proposition \ref{prop:4} from parts (1) and (2) of Theorem \ref{thm:main1}. Let $p$ be as in parts (1) and (2) of Theorem \ref{thm:main1}. Let $d$ be the rank of $J(K)$, and take $m>d$. Then for every $(x_1 ,\ldots ,x_m ) \in X\cap \Gamma $, there are integers $n_i$, not all zero, such that $n_i x_i \in J(K)$ for all $i$, and 
\[
\sum n_i x_i =0.
\]
Since $n_i x_i \in J(K)$, for all $\sigma _i \in \Gal (K)$, we have
\[
\sum n_i \sigma _i (x_i )=0.
\]
We deduce that if $n>d$, then for any $\sigma _{ij} $ in $\Gal (K)$, if $Y:=X^n$ and $A:=J^n$ then $(\sigma _{ij} (x_j ))$ is a low rank point of $Y^{d+1}\subset A^{d+1}$, i.e.
\begin{equation}\label{eqn:lowrank}
(S_n )^m \subset Y(\overline{K})^{d+1}_{\rk \leq d},
\end{equation}
where for a subvariety $V$ of an abelian variety $A^n$, and $r<n$, we define
\[
V(K)_{\rk \leq r}:=\{ (x_1, \ldots ,x_n )\in V(K):\rk \langle x_1 ,\ldots ,x_n \rangle \leq r \}
\]
where $\langle x_1 ,\ldots ,x_n \rangle $ denotes the subgroup of $A(K)$ generated by $x_1 ,\ldots ,x_n $.
\subsection{Low rank points and Coleman integrals}
By part (2) of Theorem \ref{thm:main1}, $S_n ^m \subset X^{nm}(\Q _p ^{\mathrm{nr}})$. By part (1) of Theorem \ref{thm:main1}, the image of $S_n ^m$ in $X(\overline{\F }_p )^{nm}$ is finite. Hence to prove Proposition \ref{prop:4} it is enough to prove that, for each $(\overline{z}_{ij})$ in $X(\overline{\F }_p )^{nm}$, $S_n ^m \cap ](\overline{z}_{ij})[(\Q _p ^{\mathrm{nr}})$  is not Zariski dense in $X^{nm}_{\overline{\Q }_p }$. Hence by \eqref{eqn:lowrank}, Proposition \ref{prop:4} is implied by the following lemma.
\begin{proposition}
If $mn\geq (ng-d)(m-d)$, then for any $(\overline{z}_{ij})$ in $X(\overline{\F }_p )^{nm}$,
\[ 
Y(\Q _p ^{\mathrm{nr}})^{m}_{\rk \leq d}\cap ](\overline{z}_{ij})[(\Q _p ^{\mathrm{nr}})
\]
is not Zariski dense in $Y^m _{\Q _p ^{\mathrm{nr}}}$.
\end{proposition}
 
The proof of this proposition will take up the remainder of this section. We follow the strategy of the proof of Proposition 1 of \cite{dogra}. Choose a basis $\omega _1 ,\ldots \omega _g $ of $H^0 (X,\Omega )$. We have a rigid analytic map
\[
F :]\overline{z}[(\Q _p ^{\mathrm{nr}})\to \Mat _{gn,m}(\Q _p ^{\mathrm{nr}})
\]
given by sending $(z_{ij})$ to $(\int ^{z_{ij}}\omega _k )$. Under this map, the rank $\leq d$ locus gets sent to the the subvariety $\mathbb{D}_d $ of rank $\leq d$ matrices in $\Mat _{gn,m}$, which has codimension $(ng-d)(m-d)$ (see e.g. \cite{eisenbud}). We aim to prove that the preimage of $\mathbb{D}_d$ under $F $ is not Zariski dense in $X^{mn}$. Since $F $ is a rigid analytic map on an affinoid open inside $X^{mn}$, it will be enough to prove this in the formal completion of a given point $(z_{ij})$. This will be a consequence of the Ax--Schanuel theorem for abelian varieties.
\begin{theorem}[Ax,\cite{ax1972some}]\label{thm:ax}
Let $A$ be an abelian variety over a field $K$ of characteristic zero, $Z\subset A$ an irreducible subvariety of dimension d and $V\subset \Lie (A)$ a subvariety of the Lie algebra of codimension $e$. Let 
\[
\Delta \subset \widehat{A}\times \widehat{\Lie}(A)
\]
denote the graph of the logarithm map on the formal group. Let $W$ be an irreducible component of $\Delta \cap \widehat{Z}\times \widehat{V}$ of dimension $c$. If $c>d-e$, then $\pi (W)$ is contained in the formal completion of a proper abelian subvariety of $A$, where 
\[
\pi :\widehat{A}\times \widehat{\Lie }(A)\to \widehat{A}
\]
denote the projection.
\end{theorem}
In particular, if $Z$ generates $A$, and $e\geq d$, then $\pi (Z)$ is not Zariski dense in $X$. Note that if $K$ is a $p$-adic field, and $\log :A(K)\to \Lie (A)$ denotes the Coleman integral map, then the restriction
\[
\widehat{\log }:\widehat{A}\to \widehat{\Lie }(A)
\] 
of $\log $ to formal completions of the identity is equal to the logarithm map on formal groups. Hence if $W$ is a subvariety of $\Lie (A)$ containing $0$, and $V$ is a subvariety of $A$ containing the identity, then the preimage of $\widehat{W}\subset \widehat{\Lie }(A)$ in $\widehat{V}\subset \widehat{A}$ is equal to $\pi (\Delta \cap (\widehat{V}\times \widehat{W}))$.

For the sake of completeness, we spell out how to apply this theorem in our case of interest. Take $A=J^{mn}$. Map $X^{mn}$ into $J^{mn}$ via the basepoint $(z_{ij})$ and take $Z$ to be the image of $X^{mn}$. Let $V\subset \Lie (J)^{mn}$ denote the preimage of $\mathbb{D}_d$ under the composite isomorphism
\[
\Lie (J)^{mn}\to \Lie (J)^{mn} \to \Mat _{gn,m}
\]
where the righthand map is the isomorphism determined by the basis $(\omega _i )$ and the lefthand map is given by translation by $\log (\iota (z_{ij}))$. The preimage of the formal completion of $\mathbb{D}_d$ at $\log (z_{ij})$ under the map $F $ is then identified with $\pi (\Delta \cap (\widehat{Z}\times \widehat{V}))$, hence Theorem \ref{thm:ax} implies that $F ^{-1}(\mathbb{D}_d)$ is not Zariski dense.
\bibliography{bib_ZP}
\bibliographystyle{alpha}

\end{document}